\documentclass[10pt]{amsart}
\usepackage{amssymb}
\usepackage{amsbsy}
\usepackage{amscd}
\usepackage[mathscr]{eucal}
\usepackage{verbatim}
\usepackage{version}
\usepackage{color}
\usepackage[matrix,arrow]{xy}
\usepackage{graphicx}

\def\cal{\mathcal}
\def\Bbb{\mathbb}

\newenvironment{NB}{
\color{red}{\bf NB}. \footnotesize 
}{}
\excludeversion{NB}
\newenvironment{NB2}{
\color{blue}{\bf NB}. \footnotesize
}{}

\newcommand{\Ext}{\operatorname{Ext}}
\newcommand{\Hom}{\operatorname{Hom}}

\newcommand{\im}{\operatorname{im}}

\newcommand{\rk}{\operatorname{rk}}

\newcommand{\NS}{\operatorname{NS}}

\newcommand{\Pic}{\operatorname{Pic}}
\newcommand{\ch}{\operatorname{ch}}
\newcommand{\td}{\operatorname{td}}

\newcommand{\Hilb}{\operatorname{Hilb}}

\newcommand{\Coh}{\operatorname{Coh}}

\newcommand{\Div}{\operatorname{Div}}

\newcommand{\topo}{\operatorname{top}}

\font\b=cmr10 scaled \magstep5
\def\bigzerou{\smash{\lower1.7ex\hbox{\b 0}}}

\numberwithin{equation}{section}
\theoremstyle{plain}
 \newtheorem{thm}{Theorem}[section]
 \newtheorem{lem}[thm]{Lemma}
 \newtheorem{prop}[thm]{Proposition}
 \newtheorem{cor}[thm]{Corollary}

\theoremstyle{definition}
 \newtheorem{defn}[thm]{Definition}
 
\theoremstyle{remark}
 \newtheorem{rem}[thm]{Remark}
 \newtheorem{ex}[thm]{Example}
 
\begin{document}

\title{Some moduli spaces of 1-dimensional sheaves on 
an elliptic ruled surface.}
\author{K\={o}ta Yoshioka}
\address{Department of Mathematics, Faculty of Science,
Kobe University,
Kobe, 657, Japan
}
\email{yoshioka@math.kobe-u.ac.jp}

\thanks{
The author is supported by the Grant-in-aid for 
Scientific Research (No. 18H01113), JSPS}
\keywords{elliptic ruled surfaces, stable sheaves}
%\subjclass[2010]{Primary 14D20}

\begin{abstract}
We shall study moduli spaces of stable 1-dimensional sheaves on an elliptic ruled
surface.
\end{abstract}

\maketitle
%\tableofcontents

\renewcommand{\thefootnote}{\fnsymbol{footnote}}
\footnote[0]{2010 \textit{Mathematics Subject Classification}. 
Primary 14D20.}

\section{Introduction}
Let $X$ be a smooth projective surface over ${\Bbb C}$.
Moduli spaces of stable sheaves of rank $r$ on $X$ are studied by many people.
In particular, by analysing the chamber structure on the ample cone,
topological invariants of the moduli spaces are extensively studied if $X$ is a surface 
of Kodaira dimension $-\infty$ and $r>0$ (cf. \cite{G1}, \cite{G2},
\cite{Man1}, \cite{Man2}, \cite{Mo}, \cite{Y:1}, \cite{Y:2}, \cite{Y:half-K3}). 
For the moduli spaces of stable 1-dimensional sheaves on surfaces,
topological properties are also studied for
${\Bbb P}^2$, a quadratic surface and an elliptic surface.
 In this note, we shall compute the Hodge numbers of some
moduli spaces on an elliptic ruled surface, that is,
a ${\Bbb P}^1$-bundle over an elliptic curve $C$. 

%Let $\pi:X \to {\Bbb P}^1$ be an elliptic rules surface over ${\Bbb C}$.
%Thus $X$ is a ${\Bbb P}^1$-bundle over a curve $C$ and $X$ has an elliptic fibration.
%In this note, we shall study moduli spaces of stable 1-dimensional sheaves
%$E$ on $X$ such that $\Supp(E)$ is a double covering of ${\Bbb P}^1$.

%If $X$ has a multiple fiber, then the multiplicity is 2.
%Hence there is a semi-stable vector bundle ${\cal E}$ of rank 2 on an elliptic curve and
%$X={\Bbb P}({\cal E})$, where
%$\deg {\cal E}=1$ or ${\cal E}={\cal O}_C \oplus {\cal L}$ with
%${\cal L}^{\otimes 2} \cong {\cal O}_C$. 

Let $\varpi:X \to C$ be the structure morphism of the ${\Bbb P}^1$-fibration.
Let $g$ be a fiber of $\varpi$ and $C_0$ a minimal section.
Then $\NS(X)={\Bbb Z}C_0+ {\Bbb Z}g$ with 
$(C_0 \cdot g)=1$, $(g^2)=0$ and
$(C_0^2)=-e$.
The canonical divisor
$K_X=-2C_0-e g$ satisfies $(K_X^2)=0$ and $K_X$ 
is nef if and only if $e=0,-1$.
In each cases, the nef cone is generated by $g$ and $-K_X$ \cite[Prop. 2.20, 2.21]{H}.
Let $M_H(0,\xi,\chi)$ be the moduli space of stable 1-dimensional sheaves $E$
such that $c_1(E)=\xi$ and $\chi(E)=\chi$.

\begin{thm}\label{thm:e=-1}
Let $\varpi:X \to C$ be an elliptic ruled surface with $e=-1$.
Let $(\xi,\chi)$ be a pair of $\xi \in \NS(X)$ with $(\xi \cdot K_X)<0$ 
and $\chi \in {\Bbb Z}$ with $\chi \ne 0$. 
\begin{enumerate}
\item[(1)]
Assume that $H$ is a general polarization. Then
$M_H(0,\xi,\chi)$ is a smooth projective manifold of dimension
$(\xi^2)+1$.
\item[(2)]
The Hodge numbers of $M_H(0,\xi,\chi)$ is independent
of a general choice of $H$.
If $(\xi \cdot K_X)=-1$, then the generating function is given by
\begin{equation}
\begin{split}
& \sum_{-(\xi \cdot K_X)=1}  
\left(\sum_{p,q} (-1)^{p+q} h^{p,q}(M_H(0,\xi,\chi))x^p y^q \right) q^{\frac{(\xi^2)}{4}} \\
=& (x-1)^2 (y-1)^2 q^\frac{1}{4}
\prod_{n>0}
\frac{(1-x^{-1}(x^2 y^2 q)^n)^2(1-y^{-1}(x^2 y^2 q)^n)^2(1-x(x^2 y^2 q)^n)^2(1-y(x^2 y^2 q)^n)^2}
{(1-(xy)^{-1}(x^2 y^2 q)^{\frac{n}{2}})(1-(x^2 y^2 q)^{\frac{n}{2}})^2(1-(xy)(x^2 y^2 q)^{\frac{n}{2}})}.
\end{split}
\end{equation}
\end{enumerate}
\end{thm}

The smoothness of $M_H(0,\xi,\chi)$ is an easy consequence of 
the deformation theory of a coherent sheaf. 
The computation of the Hodge numbers is our main result. 
We note that there is an elliptic fibration $\pi:X \to {\Bbb P}^1$ with three multiple fibers
of multiplicity 2.
Then the assumption $(\xi \cdot K_X)=-1$ means that 
the support $D$ of $E \in M_H(0,\xi,\chi)$ is a double cover of ${\Bbb P}^1$.
We shall also treat the case where $e=0$ (Theorem \ref{thm:e=0}). 

For the proof, we shall use Fourier-Mukai transform assocaited to the elliptic fibration
\cite{Br:1}.
Since  $D$ is a double cover of ${\Bbb P}^1$, 
the computation is reduced to the computation of Hodge numbers of the moduli spaces
of stable sheaves of rank two, which is computed in 
 \cite{G1} or \cite{Y:2}.
We also use indefinite theta function in \cite{GZ} to get the product expression.
%the computation of Hodge numbers of the moduli of stable sheaves of rank 2.
Since the Betti numbers of moduli spaces for higher rank cases are computed in \cite{Mo},
it is possible to get the Betti numbers of $M_H(0,\xi,a)$ for a general $\xi$ in principle. 

We would like to remark that
the same method works for a 9 points blow-ups $X$ of ${\Bbb P}^2$.
Thus if $-K_X$ is nef, then
by using Fourier-Mukai transforms on a rational elliptic surface
and the deformation invariance of the Hodge numbers,
we can compute the Hodge numbers from the computations
for positive rank cases.  In particular we can derive an explicit form
of Euler charactersitics of $M_H(0,\xi,\chi)$ 
from the computations in \cite{Y:half-K3}, where $(\xi \cdot K_X)=-2$.

\section{Preliminaries}

{\it Notation.}
Let $X$ be a smooth projective surface.
For two divisors $D_1,D_2$ on $X$,
$D_1 \equiv D_2$ means $D_1$ is algebraically equivalent to $D_2$.
%We denotes the algebraically equivalence class of $D_1$ by $[D_1]$.
$(D_1 \cdot D_2)$ denotes the intersection number of $D_1, D_2$. We set
$(D_1^2):=(D_1 \cdot D_1)$.

For a smooth projective variety $X$,
${\bf D}(X):={\bf D}(\Coh(X))$ denotes the bounded derived category of 
the category $\Coh(X)$ of 
coherent sheaves on $X$.
For $E \in {\bf D}(X)$, $E^{\vee}:={\bf R}{\cal H}om_{{\cal O}_X}(E,{\cal O}_X)$
denotes the derived dual of $E$. 
For the Grothendieck group $K(X)$ of $X$,
we set $K(X)_{\topo}:=K(X)/\ker \ch$, where
$\ch:K(X) \to H^*(X,{\Bbb Q})$ is the Chern character map.

For an algebraic set $Y$, $e(Y):=\sum_{p,q} \sum_k (-1)^k h^{p,q}(H^k_c(Y,{\Bbb Q}))x^p y^q$
denotes the virtual Hodge polynomial of $Y$.
If $Y$ is a smooth projective manifold, then
$e(Y)$ is the Hodge polynomial of $Y$.

Assume that $\varpi:X \to C$ is an elliptic ruled surface. Thus $C$ is an elliptic curve and 
$\varpi$ is a ${\Bbb P}^1$-bundle morphism.
$C_0$ denotes a minimal section of $\varpi$ and $g$ a fiber of $\varpi$. 
We set $e:=-(C_0^2)$.
Then we have $(g^2)=0, (g \cdot C_0)=1$ and $(C_0^2)=-e$.

\subsection{Basic facts.}
Let $X$ be a smooth projective surface.
Let $H$ be an ample divisor on $X$ and $\alpha$ a ${\Bbb Q}$-divisor
on $X$.
For a coherent sheaf $E$ on $X$,
an $\alpha$-twisted Euler characteristic $\chi_\alpha(E)$ of $E$
is defined by
$\chi(E(-\alpha))=\int_X \ch(E) e^{-\alpha} \td_X$.
Matsuki and Wentworth \cite{M-W}
defined the $\alpha$-twisted stability
of a torsion free sheaf $E$ by using twisted Hilbert polynomial
$\chi_\alpha(E(nH))$. It is generalized to 1-dimensional sheaf in \cite{Y:twist2}. 
For $(r,\xi,\chi) \in {\Bbb Z}\oplus \NS(X) \oplus {\Bbb Z}$,
$M_H^\alpha(r,\xi,\chi)$ denotes the moduli space of $\alpha$-twisted stable sheaves $E$ on $X$
with $(\rk E,c_1(E),\chi(E))=(r,\xi,\chi)$ and 
$\overline{M}_H^\alpha(r,\xi,\chi)$ the projective
compactification by adding 
$S$-equivalence classes of $\alpha$-twisted semi-stable sheaves
(see \cite{M-W} for $r>0$ and \cite[Thm. 4.7]{Y:twist2} for $r=0$).

\begin{NB}
For a pair $(H,\alpha)$ of an ample divisor $H$ and a ${\Bbb Q}$-divisor $\alpha$,
let ${\cal M}_H^\alpha({\bf e})^{ss}$ be the moduli stack of $\alpha$-twisted semi-stable sheaves
$E$ with $\tau(E)$ and ${\cal M}_H^\alpha({\bf e})^s$ the  substack of $\alpha$-twisted stable sheaves.

Let $e({\cal M}_H^\alpha({\bf e})^{ss})$ be the virtual Hodge polynomial of
${\cal M}_H^\alpha({\bf e})^{ss}$ in \cite[1.1]{Y:twist2}.
\end{NB}

For a torsion free sheaf $E$ on $X$,  we set
$$
\Delta(E)=c_2(E)-\frac{\rk E-1}{2\rk E}(c_1(E)^2) \in {\Bbb Q},
$$
where $\rk E$ is the rank of $E$.
Then $\chi(E,E)=-2\rk E \Delta(E)$ and
we have the following relations.
\begin{equation}
\begin{split}
\chi(E)= & \ch_2(E)-\frac{1}{2}(c_1(E) \cdot K_X)=
\frac{1}{2}(c_1(E) \cdot (c_1(E)-K_X))-c_2(E),\\
2 \rk E \Delta(E)= & -2 \rk E \chi(E)-\rk E (c_1 \cdot K_X)+(c_1(E)^2).
\end{split}
\end{equation}

Let $X$ be an elliptic ruled surface. 
Then $h^0({\cal O}_X)=h^1({\cal O}_X)=1$ and $h^2({\cal O}_X)=0$. In particular
$\chi({\cal O}_X)=0$. 
We have $e(X)=(1+xy)(1-x)(1-y)$.
We set
\begin{equation}
Z_{x,y}(X,u):=\frac{(1-xu)(1-yu)(1-x^2y u)(1-xy^2 u)}
{(1-u)(1-xyu)^2(1-x^2 y^2 u)}.
\end{equation}
Then the Hodge polynomials of the Hilbert schemes $\Hilb_X^n$ of $n$ points on $X$
are give by 
\begin{equation}\label{eq:e(Hilb)}
\sum_n e(\Hilb_X^n)q^n=\prod_{a >0} Z_{x,y}(X,(xy)^{-1}(xyq)^a)
\end{equation} 
(see \cite{GS}).

\begin{prop}\label{prop:deform}
Let $\varpi:X \to C$ be an elliptic ruled surface
% ${\Bbb P}^1$-bundle over an elliptic curve $C$ 
such that
$-K_X$ is nef.
Assume that $\gcd(r,\xi,\chi)=1$ and
$H$ is general with respect to $(r,\xi,\chi)$. 
\begin{enumerate}
\item[(1)]
$M_H(r,\xi,\chi)$ is a smooth projective manifold with 
$$
\dim M_H(r,\xi,\chi)=-2 r\chi-r(\xi \cdot K_X)+(\xi^2)+1
$$
unless $r=(\xi \cdot K_X)=0$.
\item[(2)]
Assume that $r \ne 0$ or $(\xi \cdot K_X) \ne 0$. Then
the deformation class of $M_H(r,\xi,\chi)$ is independent of $X$. 
In particular $e(M_H(r,\xi,\chi))$ is independent of the choice of $X$.
\end{enumerate}
\end{prop}

\begin{proof}
(1)
It is sufficiently to prove
$\Ext^2(E,E)=0$ for $E \in M_H(r,\xi,\chi)$.
If $r>0$, then $(K_X \cdot H)<0$ implies
the claim.
If $r=0$, the the claim will be proved in Proposition \ref{prop:smooth}.

(2)
\begin{NB}
Old:
Let $T$ be a smooth curve and
$\phi:{\cal E}\to {\cal O}_{T \times C}$ a family of semi-stable vector bundles ${\cal E}_t$ $(t \in T)$
of rank 2 and degree 0 or $-1$ on $C$ and surjective morphisms $\phi_t$.
Then we have a family of ${\Bbb P}^1$-bundles ${\cal X} \to C \times T \to T$ with a family of 
minimal sections
$\sigma:C \times T \to {\cal X}$.
Then $\NS({\cal X}_t)={\Bbb Z}\sigma_t + {\Bbb Z}g$ and
$H=x \sigma_t+y g$ is ample if and only if $x,y>0$.
\end{NB}
Let $T$ be a smooth curve and
$\phi:{\cal E}$ a family of semi-stable vector bundles ${\cal E}_t$ $(t \in T)$
of rank 2 and degree 0 or $-1$ on $C$
Then we have a family of ${\Bbb P}^1$-bundles ${\cal X}:={\Bbb P}^1({\cal E}) \to C \times T \to T$.
%Since $\NS({\cal X}_t)=H^2({\cal X}_t, {\Bbb Z})$ $(t \in T)$ is free of rank 2,
%replacing $T$ by a covering of $T$, we may assume that
%there is a family of line bundles ${\cal L}$ on ${\cal X}$ such that 
Let ${\cal L}:={\cal O}_{{\Bbb P}({\cal E})}(1)$ be the tautological line bundle of the projective
bundle. Then  
$\NS({\cal X}_t)={\Bbb Z}c_1({\cal L})_t + {\Bbb Z}g$,
$(c_1({\cal L})_t \cdot g)=1$ and $(c_1({\cal L})^2_t)=-e$.
For a family of ample divisors ${\cal H}=x c_1({\cal L})+y g$,
%$({\cal H}_t \cdot K_{{\cal X}_t})<0$ 
the nefness of $K_{{\cal X}_t}$ implies that 
we have a smooth family of moduli spaces
$M_{({\cal X},{\cal H})/T}(r,\xi,\chi) \to T$.
In particular $e(M_{{\cal H}_t}(r,\xi,\chi))$ is 
independent of $t \in T$.
\end{proof}

\begin{NB}
Assume that $(\xi \cdot C_0) \ne 0$.
If $r=0$, then $\chi_\alpha(E(K_X))=\chi_\alpha(E)+(K_X \cdot \xi(E))$.
\end{NB}

\begin{rem}
Since $-K_X$ is numerically equivalent to an effective divisor
for any elliptic ruled surface,
if $r>0$, then the claim (1) holds without assuming the nefness of $-K_X$. 
\end{rem}

\begin{defn}
Let $X$ be an elliptic ruled surface
% ${\Bbb P}^1$-bundle over an elliptic curve $C$ 
such that
$-K_X$ is nef.
Then $M_{-K_X+kg}(r,\xi,\chi)$ is independent of $k \gg 0$.
We set $M(r,\xi,\chi):=M_{-K_X+kg}(r,\xi,\chi)$ $(k \gg 0)$. 
\end{defn}

\subsection{Moduli of stable 1-dimensional sheaves.}

Let $X$ be a smooth projective surface.

\begin{lem}\label{lem:vanish}
Assume that $-K_X$ is nef.
For a ${\Bbb Q}$-divisor $\alpha$, let $E$ and $F$ be
$\alpha$-twisted semi-stable sheaves of dimension 1
such that
\begin{equation}
\frac{\chi_\alpha(E)}{(c_1(E) \cdot H)} \geq
\frac{\chi_\alpha(F)}{(c_1(F) \cdot H)}.
\end{equation}
Then $\Hom(E,F(K_X))=0$ if one of the following conditions hold:
\begin{enumerate}
\item[(1)]
\begin{equation}
\frac{\chi_\alpha(E)}{(c_1(E) \cdot H)} >
\frac{\chi_\alpha(F)}{(c_1(F) \cdot H)}.
\end{equation}
\item[(2)]
$-K_X$ is ample.
\item[(3)]
$E$ is $\alpha$-twisted stable and $(c_1(E) \cdot K_X) \ne 0$.
\item[(4)]
$F$ is $\alpha$-twisted stable and $(c_1(F) \cdot K_X) \ne 0$.
\end{enumerate} 
\end{lem}

\begin{proof}
Assume that
there is a non-trivial homomorphism
$\varphi:E \to F(K_X)$.
We set $G:=\varphi(E)$.
Then 
\begin{equation}
\frac{\chi_\alpha(E)}{(c_1(E) \cdot H)} \leq
\frac{\chi_\alpha(G)}{(c_1(G) \cdot H)}.
\end{equation}
Since $G(-K_X)$ is a subsheaf of $F$,
we also have
\begin{equation}
\frac{\chi_\alpha(G)-(c_1(G) \cdot K_X)}{(c_1(G) \cdot H)}=
\frac{\chi_\alpha(G(-K_X))}{(c_1(G) \cdot H)} \leq 
\frac{\chi_\alpha(F)}{(c_1(F) \cdot H)}.
\end{equation}
Since $(c_1(G) \cdot K_X) \leq 0$, we get $(c_1(G) \cdot K_X)=0$ and
\begin{equation}
\frac{\chi_\alpha(E)}{(c_1(E) \cdot H)}=
\frac{\chi_\alpha(G)}{(c_1(G) \cdot H)}= 
\frac{\chi_\alpha(F)}{(c_1(F) \cdot H)}.
\end{equation}
In particular cases (1) and (2) do not occur. 
If $E$ is $\alpha$-twisted stable, then 
$E \cong G$, which implies $(c_1(E) \cdot K_X)=0$.
If $F$ is $\alpha$-twisted stable, then 
$G \cong F$, which implies $(c_1(F) \cdot K_X)=0$.
Thus cases (3) and (4) do not occur.
Therefore $\Hom(E,F(K_X))=0$.
\end{proof}

\begin{prop}[cf. {\cite[Prop. 2.7]{Y:twist2}}]\label{prop:chamber0}
Assume that $-K_X$ is nef. Then
$e(M_H^\eta(0,\xi,\chi))$ is independent of the choice of a general $(H,\eta)$.
\end{prop}

\begin{proof}
Thanks to Lemma \ref{lem:vanish}, 
we can show that 
the claim of \cite[Prop. 2.6]{Y:twist2} holds.
Hence the claim holds (see \cite[Prop. 2.7]{Y:twist2}).
\end{proof}

\begin{rem}
If $\chi \ne 0$, then
there is a general $(H,\eta)$ with $\eta=0$
(see \cite[Lem. 1.2]{Y:7}).
\end{rem}

\begin{prop}\label{prop:smooth}
Assume that $-K_X$ is nef and $(\xi \cdot K_X) \ne 0$.
Then $\Ext^2(E,E)=0$ for $E \in M_H^\alpha(0,\xi,\chi)$.
\end{prop}

\begin{proof}
For an $\alpha$-twisted stable sheaf $E \in M_H^\alpha(0,\xi,\chi)$, 
Lemma \ref{lem:vanish} implies
$\Ext^2(E,E)=\Hom(E,E(K_X))^{\vee}=0$.
\begin{NB}
assume that
there is a non-trivial homomorphism
$\varphi:E \to E(K_X)$.
We set $G:=\varphi(E)$.
Then 
\begin{equation}
\frac{\chi_\alpha(E)}{(c_1(E) \cdot H)} \leq
\frac{\chi_\alpha(G)}{(c_1(G) \cdot H)}.
\end{equation}
Since $G(-K_X)$ is a subsheaf of $E$,
we also have
\begin{equation}
\frac{\chi_\alpha(G)-(c_1(G) \cdot K_X)}{(c_1(G) \cdot H)}=
\frac{\chi_\alpha(G(-K_X))}{(c_1(G) \cdot H)} \leq 
\frac{\chi_\alpha(E)}{(c_1(E) \cdot H)}.
\end{equation}
Since $(c_1(G) \cdot K_X) \leq 0$, we get $(c_1(G) \cdot K_X)=0$ and
\begin{equation}
\frac{\chi_\alpha(G)}{(c_1(G) \cdot H)}= 
\frac{\chi_\alpha(E)}{(c_1(E) \cdot H)}.
\end{equation}
By the stability of $E$, we see that
$\varphi$ is an isomorphism and $(c_1(E) \cdot K_X)=0$, which is a contradiction.
\end{NB}
\end{proof}

\begin{rem}
Assume that $-K_X$ is nef and $(\xi \cdot K_X)=0$.
Then $E \in M_H^\alpha(0,\xi,\chi)$ satisfies
$\chi(E(K_X))=\chi$. In this case,
$\Ext^2(E,E)=0$ if and only if $E \not \cong E(K_X)$.
\begin{NB}
If $E \not \cong E(K_X)$, then $E$ deforms along deformations of $X$.
\end{NB}
\end{rem}

For a purely 1-dimensional sheaf $E$, $\Div(E)$ denotes the scheme-theoretic support of $E$.
We have a morphism 
\begin{equation}
\begin{matrix}
M_H^\alpha(0,\xi,\chi) & \to & \Hilb_X^\xi\\
E & \mapsto & \Div(E).
\end{matrix}
\end{equation}
Assume that $h^2({\cal O}_X)=0$ and $\xi-K_X$ is ample.
Then Kodaira vanishing theorem implies $h^1({\cal O}_X(D))=0$
if $c_1({\cal O}_X(D))=\xi$. 
Hence we see that $\Hilb_X^\xi$ is smooth, by   
$h^1({\cal O}_D(D))=0$ ($D \in \Hilb_X^\xi$).
 
\begin{rem}
Let $E$ be a stable 1-dimensional sheaf with $(K_X \cdot \Div(E)) < 0$.
If $(K_X \cdot C) \leq 0$ for all irreducible components $C$ of $\Div(E)$,
then the proof of Proposition \ref{prop:smooth} implies $\Ext^2(E,E)=0$.
On the other hand we have an example of $E$ such that $\Ext^2(E,E) \ne 0$.
Thus the nefness is a reasonable condition to ensure the smoothness of the moduli space. 
\end{rem}

\begin{ex}
Assume that $X$ is a ruled surface with a minimal section $C_0$.
Assume that $e:=-(C_0^2)>0$.
Let $E$ be a stable 1-dimensional sheaf such that
$\Div(E)=C_0+C$ and $C_0$ intersects $C$ properly.
Then $L:=E_{|C_0}/(\text{torsion})$ is a line bundle on $C_0$.
Assume that $h^0({\cal O}_{C_0}(K_X-C))=h^0(K_{C_0}(-C_0-C)) \ne 0$.
We set $F:=\ker(E \to E_{|C}/(\text{torsion}))$. 
$F$ is a line bundle on $C_0$.
Then we have an injective homomorphism $L(-C) \to F$. 
Hence we get a non-zero homomorphism $E \to L \to F(C) \to F(K_X) \to E(K_X)$.
\end{ex}

\section{Moduli spaces of stable 1-dimensional sheaves on $X$.}

\subsection{An elliptic ruled surface with $e=-1$.}
Let $\varpi:X \to C$ be an elliptic ruled surface with $e=-1$. 
Thus we have $(C_0^2)=-e=1$.
Then $-K_X \equiv 2C_0-g$ and $|-2K_X|$ defines an elliptic fibration
$\pi:X \to {\Bbb P}^1$.
$\pi$ has three multiple fibers $\Pi_1,\Pi_2,\Pi_3$ of multiplicity 2
and ${\cal O}_X(K_X) \cong \pi^*({\cal O}_{{\Bbb P}^1}(-2))(\Pi_1+\Pi_2+\Pi_3)$.
We set $f_0:=\Pi_1$.
Then $f_0 \equiv -K_X \equiv 2C_0-g$.
We also have
$(f_0 \cdot C_0)=1$, $(f_0 \cdot g)=2$ and
$$
\NS(X)={\Bbb Z}C_0+{\Bbb Z}g={\Bbb Z}C_0+{\Bbb Z}f_0.
$$
%$\NS(X)={\Bbb Z}f_0+{\Bbb Z}g+{\Bbb Z}\frac{f_0+g}{2}$.

As we mentioned in \cite[sect. 0]{Y:2}, we know the Hodge numbers of
$M_H(2,c_1,\chi)$ with $(c_1 \cdot g)=1$.
Thus we get
\begin{equation}\label{eq:e-poly1}
\begin{split}
& \sum_n e(M(2,C_0-g,n))q^{-n+\frac{1}{4}} \\
=&
\frac{(x-1)^2 (y-1)^2}{xy-1}\left( 
\sum_{\substack{a \geq 0\\ 2b-a \geq 0 }} (x^2y^2 q)^{\frac{(4b+1-2a)(2a+1)}{4}}(xy)^{\frac{(4b+1-2a)}{2}}-
\sum_{\substack{a < 0\\ 2b-a <0}} (x^2y^2 q)^{\frac{(4b+1-2a)(2a+1)}{4}}(xy)^{\frac{(4b+1-2a)}{2}}
\right)\\
& \quad \times 
\prod_{a \geq 1}Z_{x,y}(X,x^{-1} y^{-1}(x^2 y^2 q)^a)^2
\end{split}
\end{equation}
\begin{NB}
For $c_1=C_0-g$,
$\xi=(a+1/2)C_0-(b+1/2)g$.
$n:=2b-a$. Then $n \equiv a \mod 2$.
\end{NB}
and
\begin{equation}\label{eq:e-poly2}
\begin{split}
& \sum_n e(M(2,C_0,n))q^{-n+\frac{3}{4}} \\
=
& \frac{(x-1)^2(y-1)^2}{xy-1}\left( 
\sum_{\substack{a \geq 0\\ 2b-a > 0}} (x^2y^2 q)^{\frac{(4b-1-2a)(2a+1)}{4}}(xy)^{\frac{(4b-1-2a)}{2}}-
\sum_{\substack{a < 0\\ 2b-a \leq 0}} (x^2y^2 q)^{\frac{(4b-1-2a)(2a+1)}{4}}(xy)^{\frac{(4b-1-2a)}{2}}
\right)\\
& \quad \times \prod_{a \geq 1}Z_{x,y}(X,x^{-1} y^{-1}(x^2 y^2 q)^a)^2
\end{split}
\end{equation}
\begin{NB}
$m:=2b-a-1$. Then $m \equiv a+1 \mod 2$.
For $c_1=g$,
$\xi=-(b+\frac{1}{2})g+af$, $\xi \cdot f=-(b+1/2)<0$.
\end{NB}
(see also the proof of \cite[Prop. 3.3]{Y:1}).
Our argument is a consequence of \cite{K} and 
the computation \eqref{eq:e(Hilb)} of $e(\Hilb_X^n)$ by G\"{o}ttsche and Soergel \cite{GS}.
G\"{o}ttsche independently computed Hodge numbers by using virtual Hodge polynomials
and chamber structures of polarizations \cite[Thm. 4.4]{G1}.

We define two injective maps
$$
\nu_i: {\Bbb Z} \times {\Bbb Z} \to {\Bbb Z} \times {\Bbb Z},\; (i=1,2)
$$
by
\begin{equation}
%\begin{split}
\nu_1(a,b):=(a,2b-a),\;
\nu_2(a,b):=(a,2b-a-1).
%\end{split} 
\end{equation}
We note that
$(a,n) \in \im \nu_1$ if and only if $n \equiv a \mod 2$
and $(a,n) \in \im \nu_2$ if and only if $n \not \equiv a \mod 2$.
Hence we get
\begin{equation}
\im \nu_1 \cap \im \nu_2= \emptyset,\; \im \nu_1 \cup \im \nu_2= {\Bbb Z} \times {\Bbb Z}.
\end{equation}
\begin{NB}
\begin{equation}
b=\begin{cases}
\frac{a+n}{2} & n \equiv a \mod 2\\
\frac{a+n+1}{2} & n \not \equiv a \mod 2.
\end{cases} 
\end{equation}
\end{NB}
Then \eqref{eq:e-poly1} and \eqref{eq:e-poly2} are expressed as
\begin{equation}\label{eq:r=2}
\begin{split}
& \sum_n e(M(2,C_0,n))q^{-n+\frac{3}{4}} +
\sum_n e(M(2,C_0-g,n))q^{-n+\frac{1}{4}}\\
=
& \frac{(x-1)^2(y-1)^2}{xy-1}\left( 
\sum_{\substack{a \geq 0\\ n  \geq  0}} (x^2y^2 q)^{\frac{(2n+1)(2a+1)}{4}}(xy)^{\frac{(2n+1)}{2}}-
\sum_{\substack{a < 0\\ n< 0}} (x^2y^2 q)^{\frac{(2n+1)(2a+1)}{4}}(xy)^{\frac{(2n+1)}{2}}
\right)\\
& \quad \times \prod_{a \geq 1}Z_{x,y}(X,x^{-1} y^{-1}(x^2 y^2 q)^a)^2.
\end{split}
\end{equation}

\subsection{An autoequivalence of $X$.}
In this subsection, we shall relate the Hodge numbers of the moduli spaces of stable 1-dimensional sheaves
to \eqref{eq:r=2} by the relative Fourier-Mukai transform of Bridgeland \cite{Br:1}.  
We note that $2f_0$ is algebraically equivalent to a smooth fiber $f$
of the elliptic fibration $\pi$.
We set
$Y:=M_H(0,2f_0,1)$.
Then $Y$ is a fine moduli space, which is a smooth projective surface.
$Y$ has an elliptic fibration $Y \to C$ which is a compactification of
the relative Picard scheme $\Pic_{X'/C'}^1 \to C'$, where
$X'=X \setminus \cup_i \Pi_i$ and $C'=\pi(X')$.
Let ${\bf P}$ be a universal family on $X \times Y$.
We set ${\bf Q}:={\bf P}^{\vee}[1]$. Then 
${\bf P}$ and ${\bf Q}$ are coherent sheaves on $X \times Y$
and they are flat over $X$ and $Y$ (\cite[Lem. 5.1]{Br:1}). 
We note that ${\bf Q}_{|X \times \{ y\}} \in M_H(0,2f_0,-1)$
for all $y \in Y$. 
Let us consider a Fourier-Mukai transform
$\Phi_{X \to Y}^{{\bf P}^{\vee}}:{\bf D}(X) \to {\bf D}(Y)$ defined by
$$
\Phi_{X \to Y}^{{\bf P}^{\vee}}(E):={\bf R}p_{Y*}({\bf P}^{\vee} \otimes p_X^*(E)),\; 
E \in {\bf D}(X) 
$$
where $p_X$ and $p_Y$ are projections from $X \times Y$ to $X$ and $Y$
respectively. 
By \cite[Thm. 1.1]{U}, we have an identification $Y \cong X$ as elliptic surfaces 
over $C$. 
We denote divisors on $Y$
corresponding to $C_0,f_0,f,g \subset X$ via the identification $X \cong Y$
by the same symbols $C_0,f_0,f,g$. 
%${\bf Q}_{| \{ x \} \times X} \in \Coh(X)$ for all $x \in X$. Thus
%${\bf Q}$ is a coherent sheaf on $X \times X$ which is flat 
%over $X$. ${\bf Q}_{|f \times X}$ is a flat family of line bundles on $f$. 
%
\begin{NB}
$(C_0 \cdot f)=2$.
We note that $C_0-f_0=g-C_0$.
For $E \in M(2,C_0+g,\Delta)$,
$E(-g)_{|f}$ is a vector bundle of rank 2 and $\chi(E(-g)_{|f})=2$.
We have $M(2,C_0+g,\Delta) \cong M(2,C_0-f_0,\Delta)$
 by $E \mapsto E(-C_0)$.

Let $\Phi:{\bf D}(X) \to {\bf D}(Y)$ be the Fourier-Mukai transform.
\end{NB}
\begin{NB}
$\Ext^1({\bf P}_{|X \times \{ x \}},{\cal O}_X)=H^1(X,{\bf P}_{|X \times \{ x \}}(K_X))^{\vee}=0$
$\Hom({\bf P}_{|X \times \{ x \}},{\cal O}_X)=0$ 
\end{NB}
We note that $\Hom({\bf P}_{|X \times \{ y\}},{\cal O}_X)=0$
and
$\Ext^1({\bf P}_{|X \times \{ y\}},{\cal O}_X)=H^0(X,{\bf Q}_{|X \times \{ y\}})=0$
for all $y \in Y$.
Hence $\Phi_{X \to Y}^{{\bf P}^{\vee}}({\cal O}_X)[2]$ is a line bundle on $Y$.
Replacing the family ${\bf P}$, we may assume that
$\Phi_{X \to Y}^{{\bf P}^{\vee}}({\cal O}_X)[2]={\cal O}_Y$.
We note that $K(X)_{\topo}$ is generated by
${\cal O}_X,{\cal O}_X(C_0),{\cal O}_X(f_0),{\Bbb C}_x$. For these generators,
we get the following.
\begin{lem}\label{lem:Phi-K(X)}
\begin{enumerate}
\item[(1)]
We have the following relation in $K(Y)_{\topo}$. 
\begin{equation}
\begin{split}
\Phi_{X \to Y}^{{\bf P}^{\vee}}({\cal O}_X) & ={\cal O}_Y\\
\Phi_{X \to Y}^{{\bf P}^{\vee}}({\cal O}_X(C_0)) &=-{\cal O}_Y(-C_0+f_0)\\
\Phi_{X \to Y}^{{\bf P}^{\vee}}({\cal O}_X(f_0)) &= {\cal O}_Y(f_0)\\
\Phi_{X \to Y}^{{\bf P}^{\vee}}({\Bbb C}_x) &=-{\cal O}_f+{\Bbb C}_x.
\end{split}
\end{equation}
%where $F$ is a smooth fiber of $\pi$.
\item[(2)]
If $\ch(E)=(r,sC_0+t f_0,a)$, then
$$
\ch \Phi_{X \to Y}^{{\bf P}^{\vee}}(E)=(r-2s,sC_0+(t-2a)f_0,a).
$$
In particuler 
$$
(c_1(E) \cdot K_X)=(c_1(\Phi_{X \to Y}^{{\bf P}^{\vee}}(E)) \cdot K_Y).
$$
\end{enumerate}
\end{lem}

\begin{proof}
We only prove (1).
Since
$\chi({\cal O}_X,{\Bbb C}_x)=\chi({\cal O}_Y[-2],{\bf P}_{|\{ x \} \times Y}^{\vee})=
-\chi({\bf Q}_{|\{ x \} \times Y})$,
we get $\chi({\bf Q}_{|\{ x \} \times Y})=-1$.
Thus ${\bf Q}_{|\{ x \} \times Y}$ is a line bundle of degree $-1$ on $f$ if $x \in f$.

Since $\Ext^2({\bf P}_{|X \times \{ y\}},{\cal O}_X(C_0))=
\Hom({\cal O}_X(C_0),{\bf P}_{|X \times \{ y\}})^{\vee}=0$ for all $y \in Y$,
$\Phi_{X \to Y}^{{\bf P}^{\vee}}({\cal O}_X(C_0))[1]$ is a line bundle on $Y$.
Thus we can write 
$\Phi_{X \to Y}^{{\bf P}^{\vee}}({\cal O}_X(C_0))[1]={\cal O}_Y(D)$ for a divisor $D$.
By 
\begin{equation}
1=\chi({\cal O}_X(C_0),{\Bbb C}_x)=\chi({\cal O}_Y(D),{\bf Q}_{|\{ x \} \times Y})=-(D \cdot f)-1,
\end{equation}
we get $(D \cdot f_0)=-1$.
Since $1=\chi({\cal O}_X(C_0))=-\chi({\cal O}_Y(D))$,
we get $D \equiv -C_0+f_0$.
By $\Phi_{X \to Y}^{{\bf P}^{\vee}}({\cal O}_X(-K_X))={\cal O}_Y(-K_Y))$,
we get $\Phi_{X \to Y}^{{\bf P}^{\vee}}({\cal O}_X(f_0)) = {\cal O}_Y(f_0)$.
 \begin{NB}
$\chi(E)=\ch_2(E)+(c_1(E) \cdot f_0)/2$.
\end{NB}
\end{proof}

\begin{prop}\label{prop:FM-isom}
For a sufficiently large $k$,
$\Phi_{X \to Y}^{{\bf P}^{\vee}[2]}$ induces an isomorphism
\begin{equation}
\begin{split}
%M(2,C_0+(k-1)f_0,a+k) \cong & M^\eta(0,4\Delta C_0+\tfrac{1-4\Delta}{2}g,a+k),
M(2p,pC_0+(l+2kp)f_0 ,\chi+kp) \cong & 
M^\eta(0,4\Delta C_0+\tfrac{p-4\Delta}{2}g,\chi+kp),
\end{split}
\end{equation} 
where
$\eta$ is a suitable ${\Bbb Q}$-divisor and 
$\dim M(2p,pC_0+(l+2kp)f_0 ,\chi+kp)=4p\Delta+1$.
\end{prop}

\begin{NB}
In the notation of \cite{PerverseII},
$\ch (G_2)^{\vee}=-\Phi_{X \to Y}^{{\bf P}^{\vee}}(0,C_0,0)$.
$\eta=-\frac{C_0}{2}$.
$\chi_\eta=\chi+kp+(C_0 \cdot(pC_0+\frac{4\Delta-p}{2}f_0))=(1+2k)p+\frac{l}{2}$.
\end{NB}

\begin{proof}
Since we use \cite[Prop. 3.4.5]{PerverseII}, let us 
explain some notations and definitions.
First of all, $G$-twisted stability is the stability 
defined by using $G$-twisted Hilbert polynomial $\chi(G^{\vee} \otimes E(nH))$
instead of the Hilbert polynomial $\chi(E(nH))$, where $G \in K(X)$
satisfies $\rk G>0$. This stability is equivalent to the 
$\alpha$-twisted stability if $\alpha=c_1(G)/\rk G$.
If $H$ is a general polarization, then the
$G$-twisted stability for a torsion free sheaf is the same as the usual Gieseker stability.
In particular $G_1$-twisted stability in \cite[Prop. 3.4.5]{PerverseII}
is the same as the usual stability.

We next explain the functor $\Psi$ and the sheaves $G_1,G_2$.
By \cite[3.2]{PerverseII},
\begin{equation}
\Psi(E)[1]={\bf R}\Hom_{p_Y}(p_X^*(E),{\bf P})[1]=
(\Phi_{X \to Y}^{{\bf P}^{\vee}}(E(K_X))[2])^{\vee}[1]=
(\Phi_{X \to Y}^{{\bf P}^{\vee}}(E)[1])^{\vee}(-K_Y).
\end{equation}
$G_1$ is a locally free sheaf on $X$ such that
$(\rk G_1,c_1(G_1))=(2(H \cdot f),H)$ (which shows $\chi(G_1,{\bf P}_{|X \times \{ y \}})=0$)
and 
$$
G_2=\Psi({\cal O}_C)[1]
=(\Phi_{X \to Y}^{{\bf P}^{\vee}}({\cal O}_C)[1])^{\vee}(-K_Y),
$$
 where 
$C \in |H|$.

We set $\alpha_2:=c_1(G_2)/\rk G_2$.
Then \cite[Prop. 3.4.5]{PerverseII} and Lemma \ref{lem:Phi-K(X)} 
imply that we have an isomorphism
\begin{equation}\label{eq:FM1}
\begin{matrix}
%M(2,C_0+(k-1)f_0,a+k) \cong & M^\eta(0,4\Delta C_0+\tfrac{1-4\Delta}{2}g,a+k),
M(2p,pC_0+(l+2kp)f_0 ,\chi+kp) & \to & 
M^{\alpha_2}(0,4\Delta C_0+\tfrac{p-4\Delta}{2}g,-(\chi+(k-1)p))\\
E & \mapsto & (\Phi_{X \to Y}^{{\bf P}^{\vee}}(E)[1])^{\vee}(-K_Y)
\end{matrix}
\end{equation} 
provided $\chi(G_2,(\Phi_{X \to Y}^{{\bf P}^{\vee}}(E)[1])^{\vee}(-K_Y))<0$.
In particular we can apply this result for a sufficiently large $k$.
For a purely 1-dimensional sheaf $F'$ on $Y$,
$F:={F'}^{\vee}(-K_Y)[1]$ is a purely 1-dimensional sheaf and
\begin{equation}
\begin{split}
\chi(G_2,F')=
\chi({F'}^{\vee},G_2^{\vee})=-
\chi(G_2^{\vee}(-2K_Y),F).
\end{split}
\end{equation}
Hence we have an isomorphism
\begin{equation}\label{eq:FMD}
\begin{matrix}
%M(2,C_0+(k-1)f_0,a+k) \cong & M^\eta(0,4\Delta C_0+\tfrac{1-4\Delta}{2}g,a+k),
M^{\eta}(0,4\Delta C_0+\tfrac{p-4\Delta}{2}g,\chi+(k-1)p) & \to & 
M^{\alpha_2}(0,4\Delta C_0+\tfrac{p-4\Delta}{2}g,-(\chi+(k-1)p))\\
F & \mapsto & F^{\vee}[1](-K_Y),
\end{matrix}
\end{equation} 
where $\eta:=-\frac{c_1(G_2)}{\rk G_2}-2K_Y$.
By the isomorphisms \eqref{eq:FM1} and \eqref{eq:FMD}, we get our claim.
\begin{NB}
For $E \in M(2p,pC_0+(l+2kp) f_0,\chi+kp)$,
we set 
$F:=\Phi_{X \to Y}^{{\bf P}^{\vee}}(E)[2]$.
Then $\Psi$ in \cite[sect. 3.2]{PerverseII} satisfies
$$
F':=\Psi(E)[1]={\bf R}\Hom_{p_Y}(p_X^*(E),{\bf P})[1]=(F(K_Y))^{\vee}[1].
$$
The condition $\chi(G_2,F')<0$ means $\chi(G_2^{\vee}(-2K_Y), F)>0$.
We set $\eta:=-\frac{c_1(G_2)}{\rk G_2}-2K_Y$.
Then the $G_2$-twisted semistability of $F'$ is equivalent to
the $\eta$-twisted semistability of $F$. 
Hence we have a desired isomorphism. 
\end{NB}
\end{proof}

\begin{NB}
\begin{equation}
\begin{split}
\chi=& \ch_2+p\\
\Delta=& -\ch_2+\frac{p+2q}{4}\\
=& p-\chi+\frac{p+2q}{4}.
\end{split}
\end{equation}
\end{NB}

\begin{rem}\label{rem:rel-Pic}
We set $\xi:=4\Delta C_0+\tfrac{p-4\Delta}{2}g$.
If $\Delta>0$, then $\xi$ is ample by \cite[Prop. 2.21]{H}.
Hence $\Hilb_X^\xi$ is smooth of dimension $p\frac{4\Delta+1}{2}$.

Assume that $p=1$ and $4\Delta \geq 3$.
We take $D \in \Hilb_X^\xi$.
For all smooth fiber $f$ of $\pi$,
$h^1({\cal O}_X(D-f))=0$. Hence $H^0(X,{\cal O}_X(D)) \to H^0(f,{\cal O}_f(D))$
is surjective. Therefore the base point of $|D|$ is a subset of multiple fibers.
By the theorem of Bertini (see \cite[Rem. 10.9.2]{H}),
a general member of $D \in \Hilb_X^\xi$ is smooth on $X \setminus \cup_{i=1}^3 \Pi_i$.
Since $(D \cdot \Pi_i)=1$, $D$ is smooth in a neighborhood of $\cup_{i=1}^3 \Pi_i$.
Therefore a general member of $D \in \Hilb_X^\xi$ is a smooth curve on $X$.
In particular a general fiber of $M_H^\alpha(0,\xi,\chi) \to \Hilb_X^\xi$ 
is the Picard variety $\Pic^\chi(D)$ of a smooth curve $D$
parameterizing line bundles $L$ on $D$ with $\chi(L)=\chi$.
\begin{NB}
$\xi=C_0+\frac{4\Delta-1}{2}f_0$ and $(\xi-2f_0)-K_X=\xi-f_0=C_0+\frac{4\Delta-3}{2}f_0$.
By Kodaira vanishing theorem, $h^1({\cal O}_X(D-F))=0$.

It is also possible to get the claim by using
$\varpi_*({\cal O}_X(nC_0))=S^n {\cal E}$ for $n>0$.
\end{NB}  
\begin{NB}
$g(D)=p\frac{4\Delta-1}{2}+1$.
$\dim\Hilb_X^D+g(D)=(D^2)+1$.

$\dim |C_0+nf_0|=n$ and $\dim |mf_0|=[\frac{m}{2}]$.
Hence a general member of $|C_0+nf_0|$ is irreducible.
\end{NB}
\end{rem}

%We note that $(4\Delta C_0+\frac{1-4\Delta}{2}g) \cdot f_0=1$,

We note that $4\Delta C_0+\frac{p-4\Delta}{2}g=pC_0+\frac{4\Delta-p}{2}f_0$.
Since $(C_0 \cdot f_0)=1$,
by using Proposition \ref{prop:FM-isom} and Proposition \ref{prop:chamber0},
we get
$$
e(M(2,C_0+lg,n))=e(M(0,C_0+\tfrac{4\Delta-1}{2}f_0,n))=
e(M_H(0,C_0+\tfrac{4\Delta-1}{2}f_0,\chi)),
$$
where $\chi$ is an arbitrary non-zero integer
and $H$ is a general polarization 
(see also Proposition \ref{prop:indep}).
Hence we get
the following result from \eqref{eq:r=2}.

\begin{lem}\label{lem:e=-1}
\begin{equation}\label{eq:e=-1}
\begin{split}
& \sum_{(\xi \cdot f_0)=1}  e(M_H(0,\xi,\chi))q^{(\xi^2)} \\
=&
\frac{(x-1)^2 (y-1)^2}{xy-1}\left( 
\sum_{\substack{a \geq 0\\ n>0  }} (x^2y^2 q)^{\frac{(2n-1)(2a+1)}{4}}(xy)^{\frac{(2n-1)}{2}}-
\sum_{\substack{a < 0\\ n \leq 0}} (x^2y^2 q)^{\frac{(2n-1)(2a+1)}{4}}(xy)^{\frac{(2n-1)}{2}}
\right)\\
& \quad \times 
\prod_{a \geq 1}Z_{x,y}(X,x^{-1} y^{-1}(x^2 y^2 q)^a)^2.
\end{split}
\end{equation}
\end{lem}

In order to get a product expression of \eqref{eq:e=-1},
we quote the following formula.

\begin{lem}[\cite{Z}, {\cite{GZ}}]\label{lem:product-formula}
\begin{equation}
\begin{split}
&
\sum_{n,m \geq 0}q^{(n+\frac{1}{2})(m+\frac{1}{2})}t^{n+\frac{1}{2}}
-\sum_{n,m < 0}q^{(n+\frac{1}{2})(m+\frac{1}{2})}t^{n+\frac{1}{2}}\\
=& \frac{\eta(q)^4}{\eta(q^{\frac{1}{2}})^2}q^{\frac{1}{8}}(t^{\frac{1}{2}}-t^{-\frac{1}{2}})
\prod_{n>0} \frac{(1-q^nt)(1-q^n t^{-1})}{(1-q^{n-\frac{1}{2}}t)(1-q^{n-\frac{1}{2}}t^{-1})}.
\end{split}
\end{equation}
\end{lem}

\begin{NB}
$\sim (t^{\frac{1}{2}}-t^{-\frac{1}{2}})q^{\frac{1}{4}}+\cdots $
\end{NB}

\begin{proof}
For a convenience sake,
we write a proof.
We set
\begin{equation}
\begin{split}
G(\tau,x,y):=& \sum_{n \geq 0, m>0} q^{nm} e^{2\pi \sqrt{-1}(-nx-my)}-
\sum_{n>0, m \geq 0}
q^{nm} e^{2\pi \sqrt{-1}(nx+my)}\\
=& \frac{\eta(\tau)^3 \theta_{11}(\tau,x+y)}{\theta_{11}(\tau,x)\theta_{11}(\tau,y)}
\end{split}
\end{equation}
(see \cite[3.1]{GZ}).
Then we see that

\begin{equation}
\begin{split}
& q^{-\frac{1}{4}}e^{2\pi \sqrt{-1}(-\frac{x}{2}+\frac{y}{2})}G(\tau,x+\tfrac{\tau}{2},y-\tfrac{\tau}{2})\\
=& q^{-\frac{1}{4}}e^{2\pi \sqrt{-1}(-\frac{x}{2}+\frac{y}{2})}
\left(\sum_{n \geq 0, m>0} q^{nm} e^{2\pi \sqrt{-1}(-n(x+\frac{\tau}{2})-m(y-\frac{\tau}{2}))}
-\sum_{n>0, m \geq 0}
q^{nm} e^{2\pi \sqrt{-1}(n(x-\frac{\tau}{2})+m(y+\frac{\tau}{2}))}
 \right)\\
=& \sum_{n \geq 0, m>0} q^{(n+\frac{1}{2})(m-\frac{1}{2})}
e^{2\pi \sqrt{-1}(-x(n+\frac{1}{2})-y(m-\frac{1}{2}))}-
\sum_{n > 0, m \geq 0} q^{(n-\frac{1}{2})(m+\frac{1}{2})}
e^{2\pi \sqrt{-1}(x(n-\frac{1}{2})+y(m+\frac{1}{2}))}\\
=&
-\left(
\sum_{n \geq 0, m \geq 0} q^{(n+\frac{1}{2})(m+\frac{1}{2})}
e^{2\pi \sqrt{-1}(x(n+\frac{1}{2})+y(m+\frac{1}{2}))}-
\sum_{n < 0, m<0} q^{(n+\frac{1}{2})(m+\frac{1}{2})}
e^{2\pi \sqrt{-1}(x(n+\frac{1}{2})+y(m+\frac{1}{2}))}
\right).
\end{split}
\end{equation}

Since
\begin{equation}
\begin{split}
q^{\frac{1}{8}}e^{2\pi \sqrt{-1} \frac{x}{2}} \theta_{11}(\tau,x+\tfrac{\tau}{2})=&
-\theta_{01}(\tau,x),\\
q^{\frac{1}{8}}e^{2\pi \sqrt{-1} \frac{y}{2}} \theta_{11}(\tau,y-\tfrac{\tau}{2})=&
\theta_{01}(\tau,y),
\end{split}
\end{equation}
we get 
\begin{NB}
Hence
\begin{equation}
\sum_{n \geq 0, m>0} q^{(n+\frac{1}{2})(m-\frac{1}{2})}
e^{2\pi \sqrt{-1}(-x(n+\frac{1}{2})-y(m-\frac{1}{2}))}-
\sum_{n > 0, m \geq 0} q^{(n-\frac{1}{2})(m+\frac{1}{2})}
e^{2\pi \sqrt{-1}(x(n-\frac{1}{2})+y(m+\frac{1}{2}))}
=-\frac{\eta(\tau)^3 \theta_{11}(\tau,x+y)}{\theta_{01}(\tau,x)\theta_{01}(\tau,y)} 
\end{equation}
\end{NB}
\begin{equation}
\sum_{n \geq 0, m \geq 0} q^{(n+\frac{1}{2})(m+\frac{1}{2})}
e^{2\pi \sqrt{-1}(x(n+\frac{1}{2})+y(m+\frac{1}{2}))}-
\sum_{n < 0, m<0} q^{(n+\frac{1}{2})(m+\frac{1}{2})}
e^{2\pi \sqrt{-1}(x(n+\frac{1}{2})+y(m+\frac{1}{2}))}
=\frac{\eta(\tau)^3 \theta_{11}(\tau,x+y)}{\theta_{01}(\tau,x)\theta_{01}(\tau,y)}. 
\end{equation}
Substituting $y=0$ and setting $t=e^{2\pi \sqrt{-1}x}$, we have
\begin{equation}
\sum_{n \geq 0, m \geq 0} q^{(n+\frac{1}{2})(m+\frac{1}{2})}
t^{(n+\frac{1}{2})}-
\sum_{n < 0, m<0} q^{(n+\frac{1}{2})(m+\frac{1}{2})}
t^{(n+\frac{1}{2})}
=\frac{\eta(\tau)^3 \theta_{11}(\tau,x)}{\theta_{01}(\tau,x)\theta_{01}(\tau,0)}.
\end{equation}
By the triple multiple formula of theta functions (cf. \cite[cf. (2.5)]{G2}),
we get our claim.
\end{proof}

{\it Proof of Theorem \ref{thm:e=-1}.}

By using Lemma \ref{lem:product-formula},
we get Theorem \ref{thm:e=-1} from Lemma \ref{lem:e=-1}:
\begin{equation}
\begin{split}
& \sum_{(\xi \cdot f_0)=1}  e(M_H(0,\xi,\chi))q^{(\xi^2)} \\
=& \frac{(x-1)^2 (y-1)^2}{xy-1}((xy)^{\frac{1}{2}}-(xy)^{-\frac{1}{2}})(x^2 y^2q)^{\frac{1}{4}}
\prod_{n>0} \frac{(1-(x^2 y^2 q)^n)^4}{(1-(xyq^{\frac{1}{2}})^n)^2}\\
& \times 
\prod_{n>0} \frac{(1-(x^2 y^2 q)^n xy)(1-(x^2 y^2 q)^n (xy)^{-1})}
{(1-(x^2 y^2 q)^{n-\frac{1}{2}}xy)(1-(x^2 y^2 q)^{n-\frac{1}{2}}(xy)^{-1})}
\prod_{a \geq 1}Z_{x,y}(X,x^{-1} y^{-1}(x^2 y^2 q)^a)^2\\
=&
(x-1)^2 (y-1)^2 q^{\frac{1}{4}}
\prod_{n>0} \frac{(1-(x^2 y^2 q)^n)^4}{(1-(xyq^{\frac{1}{2}})^n)^2}
\prod_{n>0} \frac{(1-(x^2 y^2 q)^n xy)(1-(x^2 y^2 q)^n (xy)^{-1})}
{(1-(x^2 y^2 q)^{n-\frac{1}{2}}xy)(1-(x^2 y^2 q)^{n-\frac{1}{2}}(xy)^{-1})}\\
& \times \prod_{a \geq 1}Z_{x,y}(X,x^{-1} y^{-1}(x^2 y^2 q)^a)^2\\
=& (x-1)^2 (y-1)^2 q^\frac{1}{4}
\prod_{n>0}
\frac{(1-x^{-1}(x^2 y^2 q)^n)^2(1-y^{-1}(x^2 y^2 q)^n)^2(1-x(x^2 y^2 q)^n)^2(1-y(x^2 y^2 q)^n)^2}
{(1-(xy)^{-1}(x^2 y^2 q)^{\frac{n}{2}})(1-(x^2 y^2 q)^{\frac{n}{2}})^2(1-(xy)(x^2 y^2 q)^{\frac{n}{2}})}.
\end{split}
\end{equation}
\qed

\begin{rem}\label{rem:birat}
By Theorem \ref{thm:e=-1},
we see that $h^{0,0}(M_H(0,\xi,\chi))=1$ for all $\xi$.
In particular they are irreducible.
Let $S$ be the open  subscheme of $\Hilb_X^\xi$ consisting of smooth 
curves $D$ and ${\cal D} \subset S \times X$ the universal family.
By Remark \ref{rem:rel-Pic},
$M_H(0,\xi,\chi)$ contains the relative Picard scheme $\Pic^\chi_{{\cal D}/S}$
as an open dense subscheme.
Hence the birational equivalence class of $M_H^\alpha(0,\xi,\chi)$ is
independent of the choice of $(H,\alpha)$.
\end{rem}

\subsection{An elliptic ruled surface with $e=0$.}

We shall treat the case where $e=0$.
We first assume that
%Let $C$ be an elliptic curve and
$X=C \times {\Bbb P}^1$.
%$f$ is a fiber of $\pi:X \to {\Bbb P}^1$ and
%$g$ a fiber of $X \to C$.
%$\chi({\cal O}_X)=0$.
Then the projection $\pi:X \to {\Bbb P}^1$ is an elliptic fibration
and
$K_X=-2C_0$.
We may assume that $g$ is a 0-section of $\pi$.
We set $z=C_0\cap g$.
\begin{NB}
\begin{prop}[{\cite{Y:twist2}}]
Assume that $\gcd(r,d)=1$.
Then $e(M(r,dg+kC_0,\Delta))=e(\Hilb_X^{r\Delta})e(C)$.
\end{prop}

\begin{proof}
By Proposition \ref{prop:deform}, we may assume that $X={\Bbb P}^1 \times C$.
In this case, the claim is a consequence of \cite{Y:twist2}.
\end{proof}

\end{NB}
%
%\subsection{Relative Fourier-Mukai transforms}
%
We set $Y:=M_H(0,C_0,0)$. Then there is a universal family
${\bf P}$ on $X \times Y$. We may assume that ${\bf P}_{|g \times Y} \cong {\cal O}_Y$.
Moreover we can identify $Y$ with $X$. Then we define
$C_0,g,z \subset Y$ via the identification $X \cong Y$.
Let $\Phi_{Y \to X}^{{\bf P}}:{\bf D}(Y) \to {\bf D}(X)$ be the Fourier-Mukai transform 
whose kernel is ${\bf P}$. Then we get
\begin{equation}
\begin{split}
\Phi_{Y \to X}^{{\bf P}}({\cal O}_Y)=& {\cal O}_g[-1]\\
\Phi_{Y \to X}^{{\bf P}}({\cal O}_g)=& {\cal O}_X\\
\Phi_{Y \to X}^{{\bf P}}({\cal O}_{C_0})=& {\cal O}_z[-1]\\
\Phi_{Y \to X}^{{\bf P}}({\cal O}_z)=& {\cal O}_{C_0}.
\end{split}
\end{equation}
If $\ch(E)=(x,rg+yC_0,a)$, then
$\ch(\Phi_{Y \to X}^{{\bf P}}(E))=(r,-xg+aC_0,-y)$.
Thus we get the following by using \cite[Prop. 3.4.5]{PerverseII}.
\begin{prop}
$\Phi_{X \to Y}^{{\bf P}^{\vee}[2]}$ induces an isomorphism
$M(p,aC_0,-l) \to M(0,pg+lC_0,a+p)$.
\end{prop}

By using Proposition \ref{prop:deform},
we get the following result.
\begin{thm}\label{thm:e=0}
Let $X$ be an elliptic ruled surface
% ${\Bbb P}^1$-bundle over $C$ 
with $e=0$.

\begin{equation}
\begin{split}
& \sum_n e(M_H(0,g+nC_0,1))q^\frac{n}{2} \\
=& (x-1)(y-1) \sum_n e(\Hilb_X^n)q^\frac{n}{2} \\
=& (x-1)(y-1)
\prod_{n>0} \frac{(1-x^{-1}(x^2y^2 q)^{\frac{n}{2}})(1-y^{-1}(x^2y^2 q)^{\frac{n}{2}})
(1-x(x^2y^2 q)^{\frac{n}{2}})(1-y(x^2y^2 q)^{\frac{n}{2}})}
{(1-x^{-1}y^{-1}(x^2y^2 q)^{\frac{n}{2}})(1-(x^2y^2 q)^{\frac{n}{2}})^2
(1-xy(x^2y^2 q)^{\frac{n}{2}})}.
\end{split}
\end{equation}

\begin{equation}
\begin{split}
& \sum_n e(M_H(0,2g+nC_0,3))q^n \\
=&
\frac{(x-1)^2(y-1)^2}{xy-1}\left( 
\sum_{\substack{a \geq 0\\ b>0}} (x^2y^2 q)^{b(2a+1)}(xy)^{2b}-
\sum_{\substack{a < 0\\ b<0}} (x^2y^2 q)^{b(2a+1)}(xy)^{2b}
\right)
\prod_{a \geq 1}Z_{x,y}(X,x^{-1} y^{-1}(x^2 y^2 q)^a)^2.
\end{split}
\end{equation}
\end{thm}

\begin{prop}\label{prop:indep}
Assume that $\gcd(r,n)=1$. Then
$e(M_H(0,rg+nC_0,\chi))$ is independent of the choice of $\chi$,
where $H$ is a general polarization. 
\end{prop}

\begin{proof}
Since $\gcd(r,n)=1$, there is a divisor $\eta$ such that
$(\eta \cdot (rg+nC_0))=1$. 
Then $M_H(0,rg+nC_0,1) \cong M^{\chi \eta-1}_H(0,rg+nC_0,\chi)$.
Since $e(M^{\chi \eta-1}_H(0,rg+nC_0,\chi))=e(M_H(0,rg+nC_0,\chi))$ by
Proposition \ref{prop:chamber0}, we get our claim.
\end{proof}

\begin{cor}
Assume that $\gcd(r,n)=1$. Then
$e(M(r,dC_0,n))$ is independent of the choice of $d$.
\end{cor}

\begin{rem}
By using Proposition \ref{prop:deform} and \cite[Thm. 0.2]{Y:twist2}, we have 
$$
e(M(r,pg+dC_0,n))=e(\Hilb_X^{-rn+rp+pd})e(C)
$$
if $\gcd(r,p)=1$.
\begin{NB}
$2r\Delta=-2r\chi+2r (c_1 \cdot C_0)+c_1^2=-2r\chi+2rp+2pd$.
\end{NB}
\end{rem}

\begin{NB}

$E={\cal O}_X(f)+{\cal O}_X-n{\Bbb C}_p$.

$\ch \Phi_{X \to X}^{{\cal P}}(E)=-(0,2g+nf,1)$.
Thus $\chi(\Phi_{X \to X}^{{\cal P}}(E))=1-\frac{1}{2}((2g+nf) \cdot K_X)=3$.

\begin{prop}
$\Phi_{X \to X}^{{\cal P}[1]}$ induces an isomorphism
$M(2,f,n) \to M(0,2g+nf,3)$.
It also induces an isomorphism
$M(2,0,n) \to M(0,2g+nf,4)$.
\end{prop}

\begin{NB2}
By the Riemann-Roch theorem,
$\chi({\cal O}_X(2g+nf))=\frac{1}{2} (2g+nf) \cdot (2g+nf-K_X)+\chi({\cal O}_X)=
2n+2$.
\end{NB2}

\begin{prop}
$\Hilb_X^{2g+nf}$ is smooth of dimension $2n+2$.
\end{prop}

\begin{proof}
For $D \in \Hilb_X^{2g+nf}$, we shall prove that $h^0({\cal O}_D(D))=2n+2$ and
$h^1({\cal O}_D(D))=0$.
By K\"{u}nneth formula, we get
\begin{equation}
\begin{split}
h^0({\cal O}_X(D))=& h^0({\cal O}_C(2p)) h^0({\cal O}_{{\Bbb P}^1}(n))=2(n+1),\\
h^1({\cal O}_X(D))=& h^0({\cal O}_C(2p))h^1({\cal O}_{{\Bbb P}^1}(n))+
h^1({\cal O}_C(2p))h^0({\cal O}_{{\Bbb P}^1}(n))=0,\\
h^2({\cal O}_X)=& h^1({\cal O}_C)h^1({\cal O}_{{\Bbb P}^1})=0.
\end{split}
\end{equation} 
By using the exact sequence
\begin{equation}
0 \to {\cal O}_X \to {\cal O}_X(D) \to {\cal O}_D(D) \to 0,
\end{equation}
we get $h^0({\cal O}_D(D))=2n+2$ and
$h^1({\cal O}_D(D))=0$.
\end{proof}

For $D \in |2g+nf|$, by the adjunction formula, 
\begin{equation}
2g(D)-2=(2g+nf) \cdot (2g+nf+K_X)=(2g+nf) \cdot (2g+(n-2)f)=4(n-1).
\end{equation}
Thus $g(D)=2n-1$.

We have a morphism
\begin{equation}
\varpi:M(0,2g+nf,\chi) \to \Hilb_X^{2g+nf}.
\end{equation}
For a smooth curve $D \in \Hilb_X^{2g+nf}$,
$\varpi^{-1}(D)=\Pic_D^\chi$.

\begin{equation}
\begin{split}
& \sum_n e(M(0,2g+nf,4))q^n \\
=&
\frac{(x-1)(y-1)}{((xy)^2-1)(xy-1)}
\left(
(x^2 y-1)(xy^2-1)\prod_{a \geq 1} Z_{x,y}(X,(xy)^{-2}(x^2 y^2 q)^a)
Z_{x,y}(X,(x^2 y^2 q)^a) \right. \\
& \left. -
(x-1)(y-1)\frac{(xy)^2+1}{2} \prod_{a \geq 1}Z_{x,y}(X,x^{-1} y^{-1}(x^2 y^2 q)^a)^2
\right)+\\
& \frac{(x-1)^2(y-1)^2}{xy-1}\left( 
\sum_{\substack{a >  0\\ b > 0}} (x^2y^2 q)^{2ba}(xy)^{2b}-
\sum_{\substack{a < 0\\ b<0}} (x^2y^2 q)^{2ba}(xy)^{2b}
\right)
\prod_{a \geq 1}Z_{x,y}(X,x^{-1} y^{-1}(x^2 y^2 q)^a)^2.
\end{split}
\end{equation}

\begin{NB2}
Assume that ${\cal O}_{C_0}(kC_0) \cong {\cal O}_{C_0}$
if $0<k<m$.
Then $h^0({\cal O}_X(kC_0))=h^0({\cal O}_X)=1$.
$h^0({\cal O}_X(nC_0+g))-h^0({\cal O}_C(nC_0))h^0({\cal O}_X(g))>0$.
\end{NB2}

\begin{rem}
If $X \not \cong C \times {\Bbb P}^1$, then a general member of $|g+nC_0|$
is smooth. 
\end{rem}

\end{NB}

\section{Appendix}

Assume that $X$ is an elliptic ruled surface with $e=-1$.
By using Fourier-Mukai transforms, we can derive the Hodge numbers of some
moduli spaces of stable sheaves of rank $r>0$ from Theorem \ref{thm:e=-1}
and \eqref{eq:e(Hilb)}.

\begin{thm}\label{thm:rank}
Assume that $\gcd(r,d_1)=1$.
\begin{enumerate}
\item[(1)]
If $r$ is even, then
$$
e(M(r,d_1 C_0+d_2 f_0,\chi))=e(M(0,\xi,\chi)),
$$
where $(\xi \cdot K_X)=-1$ and
$(\xi^2)=-2r\chi+r d_1+d_1^2+2d_1 d_2$.
\item[(2)]
If $r$ is odd, then 
$$
e(M(r,d_1 C_0+d_2 f_0,\chi))=e(\Hilb_X^n \times \Pic^0(X))=e(\Hilb_X^n)e(C),
$$
where $2n=-2r\chi+r d_1+d_1^2+2d_1 d_2$.
\end{enumerate}
\end{thm}

\begin{proof}
(1)
We set
$Y:=M_H(0,rf_0,d_1)$. 
\begin{NB}
$\gcd(r(f_0 \cdot C_0),d_1)=1$.
\end{NB}
Then $Y$ is a fine moduli space and $Y \cong X$.
Let $\Phi_{X \to Y}^{{\bf P}^{\vee}}:{\bf D}(X) \to {\bf D}(Y)$
be the Fourier-Mukai transform defined by a universal family ${\bf P}$.
In the same way as in Proposition \ref{prop:FM-isom},
we get the claim.
(2) is a consequence of the birational correspondence in \cite{Y:elliptic-wall}.
\end{proof}

\begin{rem}
By Remark \ref{rem:birat} and the proof of Theorem \ref{thm:rank},
we also see that $M(r,d_1 C_0+d_2 f_0,\chi)$
is birationally equivalent to $M(0,\xi,\chi)$
if $r$ is even.
Moreover by the computations in \eqref{eq:e-poly1} and \eqref{eq:e-poly2},
we see that they are projective bundles over $\Pic^0(X) \times \Pic^0(X)$.
If $r$ is odd, then we also remark that Bridgeland \cite{Br:1} proved that
$M(r,d_1 C_0+d_2 f_0,\chi)$ is birationally equivalent to
$\Hilb_X^n \times \Pic^0(X)$.
Thus the birational equivalence class of 
$M(r,d_1 C_0+d_2 f_0,\chi)$ is determined by 
$2r\chi+r d_1+d_1^2+2d_1 d_2$
if $\gcd(r,d_1)=1$.
\end{rem}

\begin{NB}
\subsection{Walls for $(0,\xi,\chi)$.}

Walls are defined by $(0,r_1 f_0,d_1)$.
\end{NB}

\begin{NB}
Then 
\begin{equation}
\begin{split}
& \sum_n e(M(2,g,n))q^n \\
=&
\frac{(x-1)(y-1)}{((xy)^2-1)(xy-1)}
\left(
(x^2 y-1)(xy^2-1)\prod_{a \geq 1} Z_{x,y}(X,(xy)^{-2}(x^2 y^2 q)^a)
Z_{x,y}(X,(x^2 y^2 q)^a) \right. \\
& \left. -
(x-1)(y-1)xy \prod_{a \geq 1}Z_{x,y}(X,x^{-1} y^{-1}(x^2 y^2 q)^a)^2
\right)+\\
& \frac{(x-1)^2(y-1)^2}{xy-1}\left( 
\sum_{\substack{a > 0\\ b \geq 0}} (x^2y^2 q)^{2(2b+1)a}(xy)^{2b+1}-
\sum_{\substack{a < 0\\ b<0}} (x^2y^2 q)^{2(2b+1)a}(xy)^{2b+1}
\right)
\prod_{a \geq 1}Z_{x,y}(X,x^{-1} y^{-1}(x^2 y^2 q)^a)^2+\\
& \frac{(x-1)^2(y-1)^2}{xy-1}\left( 
\sum_{\substack{a \geq 0\\ b > 0}} (x^2y^2 q)^{2b(2a+1)}(xy)^{2b}-
\sum_{\substack{a < 0\\ b<0}} (x^2y^2 q)^{2b(2a+1)}(xy)^{2b}
\right)
\prod_{a \geq 1}Z_{x,y}(X,x^{-1} y^{-1}(x^2 y^2 q)^a)^2.
\end{split}
\end{equation}
\begin{NB2}
For $c_1=g$,
$\xi=-(b+\frac{1}{2})g+af$ or $\xi=-(b+\frac{1}{2})g+af+\frac{g+f}{2}$.
Then $\xi \cdot f=-2(b+1/2)<0$, $\xi \cdot f=-2b<0$.
\end{NB2}

\begin{equation}
\begin{split}
& \sum_n e(M(2,0,n))q^n \\
=&
\frac{(x-1)(y-1)}{((xy)^2-1)(xy-1)}
\left(
(x^2 y-1)(xy^2-1)\prod_{a \geq 1} Z_{x,y}(X,(xy)^{-2}(x^2 y^2 q)^a)
Z_{x,y}(X,(x^2 y^2 q)^a) \right. \\
& \left. -
(x-1)(y-1)\frac{(xy)^2+1}{2} \prod_{a \geq 1}Z_{x,y}(X,x^{-1} y^{-1}(x^2 y^2 q)^a)^2
\right)+\\
& \frac{(x-1)^2(y-1)^2}{xy-1}\left( 
\sum_{\substack{a > 0\\ b > 0}} (x^2y^2 q)^{4ba}(xy)^{2b}-
\sum_{\substack{a < 0\\ b<0}} (x^2y^2 q)^{4ba}(xy)^{2b}
\right)
\prod_{a \geq 1}Z_{x,y}(X,x^{-1} y^{-1}(x^2 y^2 q)^a)^2+\\
& \frac{(x-1)^2(y-1)^2}{xy-1}\left( 
\sum_{\substack{a \geq 0\\ b \geq 0}} (x^2y^2 q)^{(2b+1)(2a+1)}(xy)^{2b+1}-
\sum_{\substack{a < 0\\ b<0}} (x^2y^2 q)^{(2b+1)(2a+1)}(xy)^{2b+1}
\right)
\prod_{a \geq 1}Z_{x,y}(X,x^{-1} y^{-1}(x^2 y^2 q)^a)^2.
\end{split}
\end{equation}

Assume that $e=0$.
Then $f=C_0$ is a fiber of $X \to {\Bbb P}^1$.

\begin{ex}
For $\Delta=1$,
we have 
\begin{equation}
0 \to {\cal O}_X(f-g) \to E \to {\cal O}_X(g) \to 0.
\end{equation}

$M(2,f,1)$ is a ${\Bbb P}^3$-bundle over $C \times C$.
\end{ex}

\begin{equation}
\begin{split}
& \sum_n e(M(2,f,n))q^n \\
=&
\frac{(x-1)^2(y-1)^2}{xy-1}\left( 
\sum_{\substack{a \geq 0\\ b>0}} (x^2y^2 q)^{b(2a+1)}(xy)^{2b}-
\sum_{\substack{a < 0\\ b<0}} (x^2y^2 q)^{b(2a+1)}(xy)^{2b}
\right)
\prod_{a \geq 1}Z_{x,y}(X,x^{-1} y^{-1}(x^2 y^2 q)^a)^2.
\end{split}
\end{equation}

\begin{equation}
\begin{split}
& \sum_n e(M(2,f+g,n))q^n \\
=&
\frac{(x-1)^2 (y-1)^2}{xy-1}\left( 
\sum_{\substack{a \geq 0\\ b \geq 0}} (x^2y^2 q)^{(2b+1)(2a+1)/2}(xy)^{2b+1}-
\sum_{\substack{a < 0\\ b<0}} (x^2y^2 q)^{(2b+1)(2a+1)/2}(xy)^{2b+1}
\right)
\prod_{a \geq 1}Z_{x,y}(X,x^{-1} y^{-1}(x^2 y^2 q)^a)^2 \\
=&
(x-1)^2 (y-1)^2 (xy+1)
\prod_{n>0} \frac{(1-(x^2 y^2 q)^{2n})^4}{(1-(x^2 y^2 q)^n)^2}
q^{\frac{1}{2}}
\prod_{n>0} \frac{(1-(x^2 y^2 q)^{2n} x^2y^2 )(1-(x^2 y^2 q)^{2n} (xy)^{-2})}
{(1-(x^2 y^2 q)^{2n-1}x^2 y^2)(1-(x^2 y^2 q)^{2n-1}(xy)^{-2})}\\
& \times \prod_{a \geq 1}Z_{x,y}(X,x^{-1} y^{-1}(x^2 y^2 q)^a)^2.
\end{split}
\end{equation}

\begin{NB2}
For $c_1=f+g$,
$\xi=-(b+\frac{1}{2})g+(a+\frac{1}{2})f$.
\end{NB2}

\begin{equation}
\begin{split}
& \sum_n e(M(2,g,n))q^n \\
=&
\frac{(x-1)(y-1)}{((xy)^2-1)(xy-1)}
\left(
(x^2 y-1)(xy^2-1)\prod_{a \geq 1} Z_{x,y}(X,(xy)^{-2}(x^2 y^2 q)^a)
Z_{x,y}(X,(x^2 y^2 q)^a) \right. \\
& \left. -
(x-1)(y-1)xy \prod_{a \geq 1}Z_{x,y}(X,x^{-1} y^{-1}(x^2 y^2 q)^a)^2
\right)+\\
& \frac{(x-1)^2(y-1)^2}{xy-1}\left( 
\sum_{\substack{a > 0\\ b \geq 0}} (x^2y^2 q)^{(2b+1)a}(xy)^{2b+1}-
\sum_{\substack{a < 0\\ b<0}} (x^2y^2 q)^{(2b+1)a}(xy)^{2b+1}
\right)
\prod_{a \geq 1}Z_{x,y}(X,x^{-1} y^{-1}(x^2 y^2 q)^a)^2.
\end{split}
\end{equation}
\begin{NB2}
For $c_1=g$,
$\xi=-(b+\frac{1}{2})g+af$, $\xi \cdot f=-(b+1/2)<0$.
\end{NB2}

 \begin{equation}
\begin{split}
& \sum_n e(M_{H_g}(2,0,n))q^n \\
=&
\frac{(x-1)(y-1)}{((xy)^2-1)(xy-1)}
\left(
(x^2 y-1)(xy^2-1)\prod_{a \geq 1} Z_{x,y}(X,(xy)^{-2}(x^2 y^2 q)^a)
Z_{x,y}(X,(x^2 y^2 q)^a) \right. \\
& \left. -
(x-1)(y-1) \prod_{a \geq 1}Z_{x,y}(X,x^{-1} y^{-1}(x^2 y^2 q)^a)^2
\right)-\frac{(x-1)^2(y-1)^2}{2(xy-1)}
\prod_{a \geq 1}Z_{x,y}(X,x^{-1} y^{-1}(x^2 y^2 q)^a)^2\\
=&
\frac{(x-1)(y-1)}{((xy)^2-1)(xy-1)}
\left(
(x^2 y-1)(xy^2-1)\prod_{a \geq 1} Z_{x,y}(X,(xy)^{-2}(x^2 y^2 q)^a)
Z_{x,y}(X,(x^2 y^2 q)^a) \right. \\
& \left. -
(x-1)(y-1)\frac{(xy)^2+1}{2} \prod_{a \geq 1}Z_{x,y}(X,x^{-1} y^{-1}(x^2 y^2 q)^a)^2
\right).
\end{split}
\end{equation}

Hence
\begin{equation}
\begin{split}
& \sum_n e(M(2,0,n))q^n \\
=&
\frac{(x-1)(y-1)}{((xy)^2-1)(xy-1)}
\left(
(x^2 y-1)(xy^2-1)\prod_{a \geq 1} Z_{x,y}(X,(xy)^{-2}(x^2 y^2 q)^a)
Z_{x,y}(X,(x^2 y^2 q)^a) \right. \\
& \left. -
(x-1)(y-1)\frac{(xy)^2+1}{2} \prod_{a \geq 1}Z_{x,y}(X,x^{-1} y^{-1}(x^2 y^2 q)^a)^2
\right)+\\
& \frac{(x-1)^2(y-1)^2}{xy-1}\left( 
\sum_{\substack{a > 0\\ b > 0}} (x^2y^2 q)^{2ba}(xy)^{2b}-
\sum_{\substack{a < 0\\ b<0}} (x^2y^2 q)^{2ba}(xy)^{2b}
\right)
\prod_{a \geq 1}Z_{x,y}(X,x^{-1} y^{-1}(x^2 y^2 q)^a)^2.
\end{split}
\end{equation}
\end{NB}

\end{document}